\documentclass[11pt,reqno]{amsart}
\usepackage{graphicx,enumerate,amssymb}
\usepackage{subfigure}
\usepackage[left=3.0cm,right=3.0cm,top=2.5cm,bottom=2.5cm,includeheadfoot]{geometry} 

\usepackage{algpseudocode}
\usepackage{algorithm}

\numberwithin{equation}{section}
\theoremstyle{plain}
\newtheorem{theorem}{Theorem}[section]
\newtheorem{lemma}[theorem]{Lemma}
\newtheorem{proposition}[theorem]{Proposition}

\theoremstyle{definition}

\theoremstyle{remark}
\newtheorem{remark}{Remark}[section]
\def\be{\begin{equation}}
\def\ee{\end{equation}}

\def\l{\langle}
\def\r{\rangle}
\def\ik{{i_k}}
\def\jk{{j_k}}
\def\p{\parallel}
\def\xls{x_{LS}}
\def\bea{\begin{eqnarray}}
\def\eea{\end{eqnarray}}
\def\hb{\hat{b}}

\newcommand {\calN}        {{\mathcal N}}

\newcommand {\calR}        {{\mathcal R}}
\newcommand{\bitem}{\begin{itemize}}
\newcommand{\eitem}{\end{itemize}}
\newcommand{\mc}[1]{\mathcal{#1}}

\newcommand {\Rm}        {{\mathbb R^m}}
\newcommand{\N}{\mathbb{N}}
\newcommand{\R}{\mathbb{R}}

\newcommand{\EE}{\mathbb{E}}

\newcommand{\bpm}{\begin{pmatrix}}
\newcommand{\epm}{\end{pmatrix}}
\newcommand{\bsm}{\left(\begin{smallmatrix}}
\newcommand{\esm}{\end{smallmatrix}\right)}
\newcommand{\T}{\top}

\newcommand{\ol}[1]{\overline{#1}}
\newcommand{\la}{\langle}
\newcommand{\ra}{\rangle}

\newcommand{\mrm}[1]{\mathrm{#1}}

\newcommand{\veps}{\varepsilon}

\newcommand{\gdw}{\Leftrightarrow}

\DeclareMathOperator{\rank}{rank}

\DeclareMathOperator{\Diag}{Diag}

\DeclareMathOperator{\argmax}{arg max}

\title[Single projection Kaczmarz Extended algorithms]{Single projection Kaczmarz Extended algorithms}
\author[Petra, Popa]{Stefania Petra, ~Constantin Popa}
\address[S.~Petra]{Image and Pattern Analysis Group, University of Heidelberg, Speyerer Str.~6, 69115 Heidelberg, Germany} 
\email{\{petra\}@math.uni-heidelberg.de}
\urladdr{iwr.ipa.uni-heidelberg.de}
\address[C.~Popa]{Faculty of Mathematics and Informatics, {\it OVIDIUS} University of Constanta, Blvd.~Mamaia~124,
900527 Constan\c{t}a, Romania}
\email{cpopa@univ-ovidius.ro}
\urladdr{http://math.univ-ovidius.ro/doc/CVpdf/cv\_C\_Poparo.pdf}
\date{\today} 


\begin{document}

\begin{abstract}
To find the least squares solution of a very large and inconsistent system of equations,
one can employ the extended Kaczmarz algorithm. This method
 simultaneously removes the error term, such that a consistent system is asymptotically obtained, and applies Kaczmarz iterations for the current approximation of this system. For random corrections of the right hand side and  Kaczmarz updates selected at random,  convergence to the least squares solution has been shown. We consider the deterministic
control strategies, and show convergence to a least squares solution when row and column updates are chosen according to the almost-cyclic or maximal-residual choice.
\end{abstract}

\maketitle
\tableofcontents

\section{Introduction}

The \emph{Kaczmarz algorithm} \cite{Kaczmarz1937} for solving linear systems of the form
\be
\label{eq:Ax=b}
A x = b,\qquad A \in \R^{m \times n},\; b \in \R^{m}
\ee
in the least-squares sense is a protoypical instance of so-called iterative \emph{row-action methods} \cite{Censor1981} that can be applied to very large systems of equations. Typical applications include image reconstruction from tomographic projections \cite{Gordon1970} -- see \cite{Censor1997} for an overview and further examples. The Kaczmarz algorithm has recently gained some renewed interest through the work \cite{Strohmer2009} where an \emph{expected} exponential convergence rate was shown for a \emph{randomized} control scheme, used to define the sequence of Kaczmarz iterations.

In view of practical applications where measurements define the vector $b$ in \eqref{eq:Ax=b}, the \emph{inconsistent case}
\begin{equation} \label{eq:b!=R(A)}
b \notin \mc{R}(A)
\end{equation}
is significant due to measurement errors and noise that most likely take $b$ outside the range $\mc{R}(A)$ of $A$. Needell \cite{Needell2010} extended to this case the analysis of \cite{Strohmer2009} and showed a similar rate of convergence to a ball around the solution to the consistent system whose radius depends on the condition number of $A$ and the perturbation of $b$.
Throughout this paper we consider the inconsistent
system 
\be
\label{eq:Ax=hatb}
Ax = \hat{b} 
\ee
after an error vector $r$ is added to the ``clean'' right side $\hat{b} $.

Popa \cite{Popa1995} introduced the \emph{extended Kaczmarz iteration} so as to achieve convergence to a least-squares solution in the inconsistent case \eqref{eq:b!=R(A)}. The basic idea is to interleave ``row-actions'' on $x$ with ``columns-actions'' on $\hat{b} $. The latter iteratively remove the spurious component of $\hat{b} $ orthogonal to $\mc{R}(A)$
\begin{equation}
r := P_{\mc{R}(A)^{\perp}}(\hat{b} ).
\end{equation}
In a very recent paper \cite{Zouzias2013} a theoretical bound of the expected convergence rate was established for a \emph{randomized} version of the \emph{extended Kaczmarz} iteration.

This line of research focusing on the convergence \emph{rate} of randomized (extended) Kaczmarz iterations also connects to earlier work on establishing \emph{convergence} of the \emph{deterministic} Kaczmarz iteration when applied to \emph{inconsistent} linear systems. The issue of \emph{cyclic convergence} in this connection was recognized early \cite{Gubin1967,Tanabe1971,Censor1983} but not resolved, as discussed next.

\vspace{0.25cm}
\noindent
\textbf{Contribution.}
The present paper has the following objective: we establish convergence of the \emph{extended} Kaczmarz iteration for a particular control scheme -- henceforth called \emph{maximal-residual control scheme} -- used to define the sequence of iterates: at each iterative step the largest residual with respect to $x$ and $\hat{b} $ determines the row- and column action to be performed as subsequent iterative step. It is evident that this scheme most aggressively aims to achieve convergence based on additional computational costs encountered when determining the maximal residuals. Convergence however was neither established in \cite{Popa1995} nor somewhere else in the literature, to our knowledge. This also holds for the application of the \emph{almost cyclic control scheme} \cite{Censor1997} to \emph{inconsistent} linear systems. Our present work also fills this gap in the literature.

\vspace{0.25cm}
\noindent
\textbf{Organization.} We recall the classical Kaczmarz algorithm in Section \ref{sec:K-alg}. We specify in Section \ref{sec:KE-algorithms} different iterative schemes based on the Kaczmarz algorithm and its deterministic and randomized extensions discussed above. This section also includes preparatory Lemmata for the convergence analysis of the maximal-residual control scheme, and the almost cyclic control scheme, established in Section \ref{sec:convergence}. 
 We conclude and indicate further directions of research in Section \ref{sec:conclusion}.




\subsection*{Notation} \label{sec:Notation}
We set $[n] = \{1,\dotsc,n\}$ for $n \in \N$. $\la \cdot,\cdot \ra$ denotes the Euclidean inner product and $\|\cdot\| = \|\cdot\|_{2} = \la \cdot,\cdot \ra^{1/2}$ the corresponding norm. 
For an $m \times n$ real matrix $A$, $A^\top$ will be its transpose and $\calR(A)$, $\calN(A)$ its  range and null space, respectively. $S^{\perp}$ will denote the orthogonal complement of some vector subspace $S \subset \R^q$, and $P_{C}$ the orthogonal projector onto some closed convex set $C$.
For given $\hat b \in \R^{m}$ and $A \in \R^{m \times n}$, we define the orthogonal decomposition
\begin{equation} \label{eq:b-dec}
\hat b = b + r,\qquad 
b \in \mc{R}(A),\qquad 
r \in \mc{R}(A)^{\perp} = \mc{N}(A^{\T}).
\end{equation}
The set of least-squares solution to problem \eqref{eq:Ax=b} is denoted by
\begin{equation} \label{eq:def-Sset}
LSS(A;\hat b) = \big\{ x \in \R^{n} \colon x = x_{LS} + \mc{N}(A),\;
A x_{LS} = P_{\mc{R}(A)}(\hat b) = b \big\}
\end{equation}
The probability simplex in $\R^{n}$ is
\begin{equation}
\Delta_{n} = \Big\{ x \in \R^{n} \colon x \geq 0,\, \sum_{i \in [n]} x_{i} = 1 \Big\}.
\end{equation}

$\|A\|_2$ denotes the spectral norm of a linear mapping $A: \R^n \to \R^m$ defined by \[\| A \|_2 = \sup_{x \neq 0} \frac{\| Ax \|_2}{\| x \|},\] and $\|A\|_{F} = (\sum_{i \in [m], j \in [n]} A_{ij}^{2})^{1/2}$ the Frobenius norm. The Moore-Penrose pseudoinverse is denoted by $A^+$.
Vectors are enumerated with superscripts $x^{i}$ and vector and matrix components with subscripts $x_{i}, A_{ij}$. Specifically, \emph{matrix rows and columns} are denoted by 
\begin{equation} \label{eq:notation-row-column}
A_{i} \quad\text{(row $i$)} 
\qquad\text{and}\qquad 
A^{j} \quad\text{(column $j$)}
\end{equation}
respectively. 
$\EE[\cdot]$ denotes the expectation operation applied to a random variable.
$\ell^1$ denotes the space of all absolutely summable sequences $(x_k)_{k\in\N}$ satisfying
$\sum_{k\in\N}|x_k|<\infty$, while $\ell_+$ will denote nonnegative sequences. 
The space of convergent sequences is denoted by $\ell_{c}$, and $\ell_{c_{0}}$ denotes the space of sequences converging to zero.


\section{The Kaczmarz Algorithm}\label{sec:K-alg}

The Kaczmarz Algorithm was first published \cite{Kaczmarz1937}. 
In it's simplest form the Kaczmarz iteration proceeds as follows:
\begin{algorithm}
\caption{Kaczmarz (K)}
\label{alg:Kacz}
  \begin{algorithmic}
    \Require{$A\in\R^{m\times n},\hat{b}\in\Rm, k_{max}\in \N$}
    \Return{Approximation to $\xls$ at bounded distance to $\xls$ (proportional to noise and condition number)}
    \Statex {\bf Initialization} $x^0 \in \R^n, k_{{\rm \max}}$ 
    \For  {$k= 1, \dots , k_{{\rm \max}}$} 
       \For  {$\ik= 1, \dots , m$} 
        \State Set
        \be
        \label{K-1}
        x^{k}= x^{k-1} -  \frac{\l x^{k-1}, A_{i_k} \r - \hat b_{i_k}}{\p A_{i_k} \p^2} A_{i_k}.
        \ee
        \EndFor 
    \EndFor  
  \end{algorithmic}
\end{algorithm}

In the field of image reconstruction it is known as ART (\emph{Algebraic Reconstruction Technique})
and independently rediscovered in \cite{Gordon1970}.
The algorithm is a particular 
\emph{Projection Onto Convex Sets (POCS)} algorithm \cite{ConvexFeasibility-96},
and can also be viewed
as a special instance of \emph{Bregman's balancing method} \cite{Bregman1965}, 
which for each $i:=(k \mod m)+1$ finds
$$x^{k+1}=x^k+\omega_k(P_{\hat H_i}(x^{k})-x^k)\ ,$$
where $P_{\hat H_i}(x^{k})$ is the orthogonal projection of $x^k$ on the
$i$-th hyperplane $\hat H_{i} = \{ x \in \R^n, \l A_{i}, x \r = \hat{b}_{i} \}$.

This sequential POCS method converges in the consistent case to a
point in the intersection of the convex sets, see
\cite[Th. 1]{Gubin1967}. However, in the inconsistent case it does not
converge, but convergence of the cyclic subsequences, called \emph{cyclic
convergence}, can be shown \cite[Th. 2]{Gubin1967}.

For the Kaczmarz algorithm (without relaxation), Kaczmarz \cite{Kaczmarz1937}
proved convergence to the unique solution of the system, provided
$A$ is square and invertible. Herman et al. showed in \cite{Herman1978}
that ART with relaxation converges in the consistent case. The case
in which no  (see also \cite{Popa2004}) solution exists has been considered by Tanabe \cite{Tanabe1971}, 
who proved convergence to a limit cycle of vectors. If, the
relaxation parameter $\omega_k$ goes to zero, the element of the
limit cycle approach the same vector. This has been considered by
Censor et al. \cite{Censor1983}, who show that the limiting single
vector is the least squares solution that is unique provided $A$
has full rank.

However, in both consistent and inconsistent case
no convergence rates existed in terms of matrix characteristics like e.g. the matrix condition number.
By considering a random row selection strategy 
a first important step was made in \cite{Strohmer2009} for the consistent, full rank case, and
expected convergence rates where obtained in term of linear algebraic characteristics of $A$.
The Randomized Kaczmarz algorithm \cite{Strohmer2009}
triggered a series a number of recent publications \cite{Needell2010,Eldar2011,Needell2014}.
The convergence of the the Randomized Kaczmarz algorithm 
was analyzed in \cite{Needell2010}. The \emph{expected} convergence to a ball of fixed radius centered at the
least squares solution was shown, \cite[Thm. 2.1 ]{Needell2010}. This radius is proportional to the norm of the additive noise scaled by the condition number,
and equals at most
\be\label{eq:max_radius}
\hat{k}(A) 
\max_{i\in[m]} \frac{|r_i|}{\| A_i \|},
\ee
where $\hat{k}(A) = \|A^+\|_2 \|A \|_F$.

The bound \eqref{eq:max_radius} shows that the randomized Kaczmarz
method performs well when the noise in inconsistent systems is small. 
The Kaczmarz method will not converge to the least squares solution of an inconsistent system, since its iterates
always lie in a single solution space  given by a single row of the
matrix $A$.

In order to overcome this problem and converge to a least squares solution 
we consider an approach first introduced by the second author in
\cite{cp98}, which  a employs
a iteratively modified right-hand side vector to  deal with the
inconsistent case. 
We show next that this strategy breaks the radius barrier of the standard method also for deterministic
row and column selection strategies, as shown before in \cite{Zouzias2013} for the random choice.

\section{Single Projection Kaczmarz Extended (KE) algorithms}\label{sec:KE-algorithms}

Algorithm \ref{alg:EK} extends Algorithm \ref{alg:Kacz} to inconsistent systems \eqref{eq:Ax=b} due to perturbations $\hat b = b + r$ of the right-hand side.

\begin{algorithm}
\caption{Single Projection Extended Kaczmarz (EK)}
\label{alg:EK}
  \begin{algorithmic}
    \Require{$A\in\R^{m\times n},\hat{b}\in\Rm, k_{max}\in \N$}
    \Return{Approximative least-squares solution}
    \Statex {\bf Initialization} $x^0 \in \R^n, y^0 = \hat{b}$, $\alpha, \omega \in (0,2);$ 
    \For  {$k= 1, \dots , k_{max}$} 
       \State  Select the index $j_k \in [n]$ and set
       \be
       \label{EK-1}
       y^{k} = y^{k-1} - \alpha {\l y^{k-1}, A^{j_k} \r} A^{j_k}.
       \ee 
    \State Update the right hand side as 
       \be
       \label{EK-2}
       \hb^{k}=\hat{b} - y^{k}.
       \ee
    \State Select the index $i_k \in [m]$ and set
        \be
        \label{EK-3}
        x^{k}= x^{k-1} -  \omega \frac{\l x^{k-1}, A_{i_k} \r - \hb^{k}_{i_k}}{\p A_{i_k} \p^2} A_{i_k}.
        \ee
    \EndFor  
  \end{algorithmic}
\end{algorithm}
The following Lemma examines how the correction step in \eqref{EK-2} affects
the perturbed hyperplanes
\be\label{eq:hatHik}
\hat H_{\ik}=\{x \colon \l A_{\ik} , x \r =\hat b^k_{\ik} = b_{\ik} +r_{\ik} -y^k_{\ik}\}
\ee
in view of the unperturbed hyperplanes
\be\label{eq:Hik}
 H_{\ik}=\{x \colon \l A_{\ik} , x\r = b_{\ik} \}.
\ee

\begin{lemma}\label{lemma:shift} Consider $\hat H_{\ik}$ and $ H_{\ik}$ defined by
\eqref{eq:hatHik} and \eqref{eq:Hik}. Then 
\be\label{def:delta}
\hat H_{\ik}=\{x + \gamma_{\ik} \colon x \in H_{\ik}\}
\qquad\text{where}\qquad
\gamma_{\ik}=\delta_{\ik} A_\ik,\quad
\delta_\ik=\frac{r_{\ik}-y^k_{\ik}}{\|A_\ik\|^2}.
\ee
\end{lemma}
\begin{proof} Denote $i:=\ik$ for simplicity. For $x \in H_{i}$, we have $\l A_i,x+ \delta_{i} A_i \r=
\l A_i,x\r + \delta_i\|A_i\|^2 = b_i+ r_i - y_i^{k}=\hat b^k_i$. Thus, $x+ \delta_{i} A_i \in \hat H_{i} $.
Conversely, choose  $\hat x\in \hat H_{\ik}$ arbitrary and set $x=\hat x-\delta_{i} A_i$. Then
$\l A_i,x \r = \l A_i,\hat x \r - r_i +y_i^{k} = b_i+r_i -y_i^{k} - r_i +y_i^{k} = b_i$ holds.
Consequently $x\in H_{i}$.
\end{proof}

\begin{remark}
We observe that due to the initialization $y^{0}=\hat b$ of Algorithm \ref{alg:EK}, the decomposition \eqref{eq:b-dec} and the update rule \eqref{EK-1}, it always holds that
\begin{equation} \label{eq:y-r-in-RA}
y^{k} - r \in \mc{R}(A),\qquad \forall k \in \N.
\end{equation}
\end{remark}


\subsection{Control Sequences}\label{sec:control-sequences}

We will consider the following basic deterministic control sequences, cf. \cite{Censor1981}, besides randomized control sequences \cite{Strohmer2009,Zouzias2013}.
\begin{description}
\item[Cyclic control] Set $i_{k} = k \,\mrm{mod}\, m + 1$, 
$j_{k} = k \,\mrm{mod}\, n + 1$.

\item[Almost cyclic control] Select $i_{k} \in [m],\, j_{k} \in [n]$, such there exist integers $m_{0}, n_{0}$ with
\be \label{eq:ik_AC}
[m] \subseteq \{i_{k+1},\dotsc,i_{k+m_{0}}\} 
\ee
\be\label{eq:jk_AC}
[n] \subseteq \{j_{k+1},\dotsc,j_{k+n_{0}}\}, 
\ee
for every $k \in \N$.
\item[Set-based control] Select $\jk\in[n]$ and $\ik\in[m]$ such that
\be\label{eq:jk_MR}
j_{k} = \underset{j \in [n]}{\argmax} \, |\la A^{j},y^{k-1} \ra|,
\ee
\be\label{eq:ik_MR}
i_{k} = \underset{i \in [m]}{\argmax} \,
|\la A_{i},x^{k-1} \ra - \hat b_{i}^{k}|.
\ee
Note that these sequences depend on each other through \eqref{EK-1}--\eqref{EK-3}.
Sequence $(j_{k})_{k \in \N}$ relates to largest component $\|P_{\calR (A^{j})}(y^{k})\|$ of $y^{k}$ weighted by $\|A^{j}\|$,  whereas the sequence $(i_{k})_{k \in \N}$ relates to the largest distance of $x^{k}$, weighted by $\|A_{i}\|$, to the hyperplane defined by some row $A_{i}$ and the right-hand side $\hat b^{k}$, that is updated due to \eqref{EK-2}.
\item[Random control] Define the discrete distributions 
\begin{equation}\label{eq:def-pq-distributions}
p \in \Delta_{m},\quad p_{i} = \frac{\|A_{i}\|^{2}}{\|A\|^{2}_{F}}, 
\quad i \in [m],
\qquad\qquad
q \in \Delta_{n},\quad q_{j} = \frac{\|A^{j}\|^{2}}{\|A\|^{2}_{F}},
\quad j \in [n],
\end{equation}
and sample in each step $k$ of the iteration \eqref{EK-1}
\be\label{jk_rand}
 j_{k} \sim q
\ee
 
and each step $k$ of the iteration \eqref{EK-3}
\be\label{ik_rand}
i_{k} \sim p .
\ee 
\end{description}

\begin{remark}
\label{9may}
We note that the cyclic control is a special case of the almost cyclic control.
The maximal residual choice is also known as \emph{remote set control} \cite{Censor1981}. 
\end{remark}

\subsection{The Randomized Extended Kaczmarz Algorithm}

In the recent paper \cite{Zouzias2013}, authors considered Algorithm \ref{alg:REK}
with a random selection of the indices $\jk$ and $\ik$ and $\alpha=\omega=1$.
\begin{algorithm}
  \caption{Randomized Extended Kaczmarz Algorithm (REK)}
  \label{alg:REK}
  \begin{algorithmic}
    \Require{$A\in\R^{m\times n},\hat{b}\in\Rm, k_{max}\in \N$}
    \Return{Approximative least-squares solution}
    \Statex {\bf Initialization} $x^0 \in \R^n, y^0 = \hat{b}$, $\alpha, ~\omega \in [0, 2]$ 
    \For  {$k= 1, \dots , k_{max}$} 
       \State  Select the index $j_k \in [n]$  \emph{randomly} according to \eqref{jk_rand}
    \State and set
       \be
       \label{REK-1}
       y^{k} = y^{k-1} - \alpha {\l y^{k-1}, A^{j_k} \r} A^{j_k}.
       \ee 
    \State Update the right hand side as 
       \be
       \label{REK-2}
       \hb^{k}=\hat{b} - y^{k}.
       \ee
    \State Select the index $i_k \in [m]$ \emph{randomly} according to \eqref{ik_rand}
    \State and set
        \be
        \label{REK-3}
        x^{k}= x^{k-1} - \omega \frac{\l x^{k-1}, A_{i_k} \r - \hb^{k}_{i_k}}{\p A_{i_k} \p^2} A_{i_k}.
        \ee
    \EndFor  
  \end{algorithmic}
\end{algorithm}
They proved the following convergence result along with a convergence rate. 
\begin{theorem}
\label{thm:REK}
 For any $A$, $\hat b$, and $x^0=0$, the sequence $(x^k)_{k \in\N}$ generated by  REK Algorithm 
 \ref{alg:REK}  with $\alpha=\omega=1$ converges  in expectation  to the minimal norm solution $\xls$ of \eqref{eq:Ax=hatb}, with 
 the asymptotic error reduction factor 
$$
\EE \big[\|x^k - \xls \|\big] ~\leq~ \bigg( 1 - \frac{1}{\hat{k}^2(A)} \bigg)^{\lfloor k/2\rfloor} (1 + 2 k^2(A)) \|x_{LS} \|^2,
$$
where $\hat{k}(A) = \|A^+\|_2 \|A \|_F$ and $k(A) = \sigma_1 / \sigma_\texttt{r}$, where $\sigma_1 \geq \sigma_2 \geq \dots \geq \sigma_\texttt{r} > 0$ are the nonzero singular values of $A$ and $\texttt{r}=\rank(A)$.
\end{theorem}

\subsection{The MREK Algorithm}

In this subsection we will show that $\|\gamma_\ik\|$ from \eqref{def:delta} decays geometrically 
for the maximal residual choice $\ik$ from \eqref{eq:ik_MR} 
and, in particular, that the error norms are absolutely summable.
These results will be in turn used to establish convergence of the MREK algorithm in Section \ref{sec:convergence}. 
We first collect some facts and state a basic assumption. 
{For any {invertible}  matrix $D \in \R^{n \times n}$  we have (cf.~\eqref{eq:def-Sset})
\begin{equation}
x \in LSS(A;\hat b) \qquad\gdw\qquad 
D^{-1} x \in LSS(A D;\hat b).
\end{equation}
As a consequence, by choosing $D = \Diag(\|A^{1}||^{-1},\dotsc,\|A^{n}\|^{-1})$, we may assume w.l.o.g.~that
\begin{equation} \label{eq:cols-normalized}
\|A^{j}\| = 1,\qquad j \in [n].
\end{equation}
First we need a preparatory result, which can be easily proved, see e.g.~\cite{Ansorge1984}.

\begin{lemma}
\label{lem1}
Let $\alpha \in (0, 2)$, $\delta \geq 0$ be defined by 
\be
\label{10}
\delta=\inf \{ \p A^{\T} \zeta \p, \;\zeta \in \calN(A^{\T})^{\perp}=\calR(A), \;\p \zeta \p = 1 \},
\ee
and let $A=U \Sigma V^{\T}, ~~\Sigma=diag(\sigma_1, \dots, \sigma_r, 0, \dots, 0)$ be 
a singular value decomposition of $A$, where $r=\rank(A)$. Then
\be
\label{12}
0 ~<~ \delta ~~=~~ \sigma_r ~~\leq~~ \sigma_1.
\ee
\end{lemma}

\begin{remark}
\label{newr} 
Since $\sigma_{i}^{2},\, i \in [r]$, are the eigenvalues of $A^{\T} A$ and $A A^{\T}$, respectively, scaling of $A \leftarrow \frac{1}{c} A$ by some factor $c > 0$ scales the singular values $\sigma_{i} \leftarrow \sigma_{i}/c$ as well.
Thus, by scaling the linear system \eqref{eq:Ax=b},
$$
\p Ax - \hat{b} \p = \min! \quad\Leftrightarrow\quad \p \frac{\sqrt{n}}{\sigma_r} Ax - \frac{\sqrt{n}}{\sigma_r} \hat{b} \p = \min!,
$$
we may assume w.l.o.g.~that $\delta \sqrt{\alpha (2 - \alpha)}  \leq \sqrt{n}$, hence
\be
\label{17}
1 - \frac{\delta^2 \alpha (2 - \alpha)}{n} \in [0, 1).
\ee
\end{remark}


\begin{lemma}
\label{lem:gammaMREK}
Let $\alpha \in (0, 2)$,  $k \in\N$ denote an arbitrary fixed number of iterations of  Algorithm \ref{alg:MREK}, 
with $i_k$ selected according to the maximal residual choice \eqref{eq:ik_MR}, and let
$\delta_\ik \in \R $ and $ \gamma_\ik \in \R^n$ be given by
\eqref{def:delta}.
Then, 
\begin{itemize}
\item[(i)]
there exist $M \geq 0$ and $\gamma \in [0, 1)$, independent on $k$, such that 
\be
\label{eq:gammaMREK}
\p \gamma_\ik \p ~~\leq~~ M \gamma^k,
\ee
\item[(ii)] 
$(\p \gamma_\ik \p^2)_{k\in\N} \in \ell_+ \cap \ell^1$,
\item[(iii)] 
$y^{k} \to r$ for $k \to \infty$, with $r$ given by \eqref{eq:b-dec}.
\end{itemize}
\end{lemma}
\begin{proof} 
\begin{enumerate}[(i)]
\item
Update rule \eqref{MREK-2} yields
\[
y^{k}-r = (y^{k-1}-r) - \alpha \la y^{k-1}, A^{j_{k}} \ra A^{j_{k}}.
\]
Using $y^{k}-r \in \mc{R}(A),\, \forall k$, and $\mc{R}(A)^{\perp} \ni r \perp A^{j},\, j \in [n]$, we compute
\begin{equation} \label{yk-r-2}
\|y^{k}-r\|^{2} = \|y^{k-1}-r\|^{2} - \alpha (2 - \alpha) \la y^{k-1}, A^{j_{k}} \ra^{2}.
\end{equation}
Based on property \eqref{MREK-1} defining the index $j_{k} \in [n]$, we upper bound
\[
\|y^{k}-r\|^{2} \leq \|y^{k-1}-r\|^{2} - \frac{\alpha (2 - \alpha)}{n} 
\sum_{j \in [n]} \la y^{k-1}, A^{j} \ra^{2}.
\]
Exploiting again $r \perp A^{j},\, j \in [n]$, and Lemma \ref{lem1}, we obtain
\begin{align*}
\|y^{k}-r\|^{2} &\leq \|y^{k-1}-r\|^{2} - \frac{\alpha (2 - \alpha)}{n} 
\sum_{j \in [n]} \la y^{k-1}-r, A^{j} \ra^{2} \\
&= \|y^{k-1}-r\|^{2}\bigg(1 - \frac{\alpha (2 - \alpha)}{n} 
\Big\| A^{\T} \frac{y^{k-1}-r}{\|y^{k-1}-r\|} \Big\|^{2} \bigg)
\leq \Big(1-\frac{\delta^{2} \alpha (2 - \alpha)}{n}\Big) \|y^{k-1}-r\|^{2} \\
&\leq \Big(1-\frac{\delta^{2} \alpha (2 - \alpha)}{n}\Big)^{k} \|y^{0}-r\|^{2}.
\end{align*}
Thus, with $y^{0}-r = b$,
\begin{equation} \label{MGamma-MREK}
\|\gamma_{i_{k}}\| = \frac{1}{\|A_{i_{k}}\|} |(r-y^{k})_{i_{k}}|
\leq \Big(1-\frac{\delta^{2} \alpha (2 - \alpha)}{n}\Big)^{k/2} 
\frac{\|b\|}{\min_{i \in [m]} \|A_{i}\|} =: \gamma^{k} M,
\end{equation}
\be
\label{NEW1}
\gamma = \left ( 1 - \frac{\delta^2 \alpha (2 - \alpha)}{n} \right )^{1/2}, 
\ee
and $\gamma \in [0,1)$ due to \eqref{17}.
\item
Using \eqref{MGamma-MREK}, $\gamma\in[0,1)$ and convergence of geometric series, we get
\[
\sum_{k\in\N} \p \gamma_\ik \p^2 \le \sum_{k\in\N} M^{2} \gamma^{2 k} 
= \frac{M^2}{1-\gamma^2}<\infty.
\]
\item
The derivation of (i) shows that relation \eqref{MGamma-MREK} is valid for every $i \in [m]$. Hence, since $\gamma \in [0,1)$,
\[
\|r-y^{k}\|_{\infty} \leq const.~\gamma^{k} \quad\to\quad 0
\qquad \text{for}\quad k \to \infty.
\]
\end{enumerate}
\end{proof}

\begin{algorithm}
\caption{Algorithm Maximal Residual Extended Kaczmarz (MREK)}
\label{alg:MREK}
\begin{algorithmic}
\Require
{$A\in\R^{m\times n},\hat{b}\in\Rm, k_{max}\in \N$}\\
\Return
{Approximative least-squares solution}
\Statex 
{\bf Initialization} $x^0 \in \R^n, y^0 = \hat{b}; \alpha, ~\omega \in [0, 2]$ \\
\For  
{$k= 1, \dots , k_{max}$} \\
\State  
Select the index $j_k \in [n]$ such that \\
       \be
       \label{MREK-1}
       | \l A^{j_k}, y^{k-1} \r | ~~\geq~~  | \l A^{j}, y^{k-1} \r |, ~~\forall j\in [n],
       \ee
\State and set
       \be
       \label{MREK-2}
       y^{k} = y^{k-1} - \alpha {\l y^{k-1}, A^{j_k} \r} A^{j_k}.
       \ee 
\State
 Update the right hand side as 
       \be
       \label{MREK-3}
       \hb^{k}=\hat{b} - y^{k}.
       \ee
\State 
Select the index $i_k \in [m]$ such that 
       \be
       \label{MREK-4}
       |\l A_{i_k}, x^{k-1} \r - \hb^{k}_{i_k}| ~~\geq~~ | \l A_{i}, x^{k-1} \r - \hb^{k}_{i}|, ~~\forall~ i\in [m],
        \ee
\State and set
        \be
        \label{MREK-5}
        x^{k}= x^{k-1} - \omega  \frac{\l x^{k-1}, A_{i_k} \r - \hb^{k}_{i_k}}{\p A_{i_k} \p^2} A_{i_k}.
        \ee
\EndFor  
\end{algorithmic}
\end{algorithm}
\subsection{The ACEK Algorithm}
In this section we will establish a result analogous to Lemma \ref{lem:gammaMREK} for Algorithm \ref{alg:ACEK} that corresponds to Algorithm \ref{alg:EK} in the case of the almost cyclic index selection scheme. First of all, related to (\ref{ACEK-1}) we introduce the notations
\be
\label{n1}
\varphi_j(y)=y -  \frac{\l y, A^{j} \r}{\p A^{j} \p^2} A^{j}, ~~\varphi^{\alpha}_j(y)= y - \alpha \frac{\l y, A^{j} \r}{\p A^{j} \p^2} A^{j},
\ee
and observe that the application $\varphi^{\alpha}_j$ is no more a projection and we have the equalities  
\be
\label{n2}
\varphi^{\alpha}_j(y) = ((1-\alpha)I + \alpha \varphi_j) (y).
\ee
We will replay below Lemma 21 from \cite{cp98} (see also \cite{natt}) with respect to the above applications.
\begin{lemma}
\label{lem21}
For any $\alpha \in (0,2)$, $y \in \R^m$, $j=1,\dots,n$ 
 the following are true
\bea
\p \varphi^{\alpha}_j \p &\leq& 1,
\label{eq68}\\
\p \varphi^{\alpha}_j y \p^2 - \p y \p^2 &=& (2- \alpha) \alpha
(\p \varphi_j y \p^2 - \p y \p^2).
\label{eq69}
\eea
\end{lemma}
We can now state the result analogous to Lemma \ref{lem:gammaMREK}.
\begin{lemma}
\label{lem:gammaACEK}
Let $k \geq n_{0} \in\N$ denote an arbitrary fixed number of iterations of  Algorithm \ref{alg:ACEK} with $n_{0}$ defined by \eqref{eq:jk_AC}, and 
with $i_k, j_k$ selected according to the almost cyclic choice \eqref{eq:ik_AC} and \eqref{eq:jk_AC}, respectively. Let
$\delta_\ik \in \R $ and $ \gamma_\ik \in \R^n$ be given by
\eqref{def:delta}.
Then, 
\begin{itemize}
\item[(i)]
there exist $M \geq 0$ and $\gamma \in [0, 1)$, independent on $k$, such that 
\be
\label{eq:gammaACEK}
\p \gamma_\ik \p ~~\leq~~ M \gamma^{n},
\ee
with $k = n \cdot n_{0} + l_{0},\,n \in \N,\, n_{0} \geq l_{0} \in \N_{0}$,
\item[(ii)] 
$(\p \gamma_\ik \p^2)_{k\in\N} \in \ell_+ \cap \ell^1$,
\item[(iii)] 
$y^{k} \to r$ for $k \to \infty$, with $r$ given by \eqref{eq:b-dec}.
\end{itemize}
\end{lemma}
\begin{proof} {\bf Step 1.} 
Let $k \geq 0$ be an arbitrary fixed fixed iteration of the algorithm ACEK, $J=\{ 1, \dots, n \}$, $J_k=\{ j_{k+1}, \dots, j_{k+\Gamma}\}$ and (see (\ref{ACEK-1}))
\be
\label{eq70}
y^{k+\Gamma}=\varphi^{\alpha}_{j_{k+\Gamma}} \circ \dots \circ \varphi^{\alpha}_{j_{k+1}} y^k.
\ee
 We will first show that it exists $\hat{\gamma} \in [0, 1)$ such that 
\be
\label{1-1}
\p y^{k+\Gamma} -  P_{\mc{N}(A^T)}(y^k) \p \leq \hat{\gamma} \p y^k - P_{\mc{N}(A^T)}(y^k) \p.
\ee
From (\ref{eq70}) we get
\be
\label{1-3}
y^{k+\Gamma} = \Phi^{\alpha}_k y^k, ~~{\rm where}~~ \Phi^{\alpha}_k = \varphi^{\alpha}_{j_{k+\Gamma}} \circ \dots \circ \varphi^{\alpha}_{j_{k+1}} :\R^m \longrightarrow \R^m.
\ee
Let $A_{(k)}$ be the $n \times \Gamma$ matrix defined by
\be
\label{1-3p}
A_{(k)} = col( A^{j_{k+1}}, \dots, A^{j_{k+\Gamma}}).
\ee
Because the additional $\Gamma - n$ columns of $A_{(k)}$ are among the columns of the initial matrix $A$ (see (\ref{eq:jk_AC})), we have
\be
\label{1-3pp}
\mc{N}(A^T_{(k)}) = \mc{N}(A^T), ~~{\rm thus}~~ P_{\mc{N}(A^T_{(k)})} = P_{\mc{N}(A^T)}.
\ee
If we define $\tilde{\Phi}^{\alpha}_k = \Phi^{\alpha}_k P_{\mc{R}(A)}$ we know that (see e.g. \cite{cp98})
\be
\label{1-4}
\Phi^{\alpha}_k = \tilde{\Phi}^{\alpha}_k + P_{\mc{N}(A^T)}, ~\tilde{\Phi}^{\alpha}_k  P_{\mc{N}(A^T)} = P_{\mc{N}(A^T)} \tilde{\Phi}^{\alpha}_k = 0, ~~\p \tilde{\Phi}^{\alpha}_k \p < 1. 
\ee
Then
$
y^{k+\Gamma} = \Phi^{\alpha}_k(y^k) = $ $\tilde{\Phi}^{\alpha}_k(y^k) + P_{\mc{N}(A^T)}(y^k),
$ thus
$$
\p y^{k+\Gamma} - P_{\mc{N}(A^T)}(y^k) \p = \p \tilde{\Phi}^{\alpha}_k(y^k) \p =  \p \tilde{\Phi}^{\alpha}_k \left( y^k - P_{\mc{N}(A^T)}(y^k) \right)) \p \leq 
$$
$$ 
\p \tilde{\Phi}^{\alpha}_k \p \cdot \p y^k - P_{\mc{N}(A^T)}(y^k) \p.
$$
The set $J_k \setminus J$ has at most $\Gamma - n$ elements which are among the indices from $J$. It results that there are finitely many matrices  $A_{(k)}$, thus finitely many applications $\tilde{\Phi}^{\alpha}_k$, i.e. 
\be
\label{1-6}
\hat{\gamma} = \max_{k \geq 0} \p \tilde{\Phi}^{\alpha}_k \p ~~{\rm belongs ~to}~~ [0, 1),
\ee
which gives us (\ref{1-1}). \\
{\bf Step 2.} We will now show that  it  exists $\widehat{M} \geq 0$, independent on $k$ such that 
\be
\label{1-19}
\p \gamma_k \p ~~\leq~~ \widehat{M} \hat{\gamma}^{\frac{k - k({\rm mod}~ \Gamma)}{\Gamma}},
\ee
with $\hat{\gamma}$ from (\ref{1-6}).
 From (\ref{eq:b-dec}) and (\ref{ACEK-1}) it results that $y^k - r \in \mc{R}(A), \forall k \geq 0$, i.e. $P_{\mc{N}(A^T)}(y^k) = r, \forall k \geq 0.$ Thus, 
\be
\label{11-8}
\p y^{k+\Gamma} - r \p \leq \hat{\gamma} \p y^k - r \p, ~\forall k \geq 0,
\ee
and recursively
\be
\label{11-9}
\p y^{\mu \Gamma} - r \p \leq \hat{\gamma} \p y^{(\mu - 1) \Gamma} - r \p, ~\forall \mu \geq 1.
\ee
For the arbitrary fixed index $k \geq 0$, let $\mu$ be the integer given by
\be
\label{11-9p}
\mu = \frac{k - k({\rm mod}~ \Gamma)}{\Gamma}, ~~{\rm i.e.}
\ee
\be
\label{11-10}
k = \mu \Gamma + q, ~{\rm for ~some} ~q \in \{ 0, 1, \dots, \Gamma-1 \}.
\ee
If we define $\widetilde{M}$ as
\be
\label{11-11}
\widetilde{M} = \max \{ \p y^{\Gamma-1} - r \p, \dots, \p y^0 - r \p \},
\ee
 from (\ref{11-8}) - (\ref{11-11})  we get  for any $~\mu \geq 1$
$$
\p y^k - r \p ~=~ \p y^{\mu \Gamma + q} - r \p ~\leq~ \hat{\gamma} \p y^{(\mu - 1) \Gamma + q} - r \p ~\leq~ \dots ~\leq 
$$
\be
\label{11-12}
\hat{\gamma}^{\mu} \p y^q - r \p ~\leq~ \widetilde{M} \hat{\gamma}^{\mu}.
\ee
Hence
\be
\label{11-12p}
\p \gamma_{\ik} \p = \frac{|r_{\ik}-y^k_{\ik}|}{\p A_{\ik} \p} \leq \frac{\p y^k - r \p}{\min_{i=1, \dots,m} \{ \p A_i \p \}} \leq
\hat{\gamma}^{\mu} \frac{\widetilde{M}}{\min_{i=1, \dots,m} \{ \p A_i \p \}},
\ee
which is exactly (\ref{1-19}), with $\widehat{M} = \frac{\widetilde{M}}{\min_{i=1, \dots,m} \{ \p A_i \p \}}$.\\
{\bf step 3.} Then, relation (\ref{eq:gammaACEK}) holds directly from (\ref{1-19}) and gives us also the conclusion (ii). Conclusion (iii) holds from (\ref{1-1}) and the proof is complete.
\end{proof}

\begin{algorithm}
  \caption{Algorithm Almost Cyclic Extended Kaczmarz (ACEK)}
  \label{alg:ACEK}
  \begin{algorithmic}
    \Require{$A\in\R^{m\times n},\hat{b}\in\Rm, k_{max}\in \N$, $\alpha \neq 0, \omega \neq 0$}
    \Return{Approximative least-squares solution}
    \Statex {\bf Initialization} $x^0 \in \R^n, y^0 = \hat{b};$ 
    \For  {$k= 1, \dots , k_{max}$} 
       \State  Select the index $j_k \in [n]$ in an \emph{almost cyclic} way according to \eqref{eq:jk_AC}
    \State and set
       \be
       \label{ACEK-1}
       y^{k} = y^{k-1} - \alpha \frac{{\l y^{k-1}, A^{j_k} \r}}{\p A^{j_k} \p^2} A^{j_k}.
       \ee 
    \State Update the right hand side as 
       \be
       \label{ACEK-2}
       \hb^{k}=\hat{b} - y^{k}.
       \ee
    \State Select the index $i_k \in [m]$ in an \emph{almost cyclic} way according to \eqref{eq:ik_AC}
    \State and set
        \be
        \label{ACEK-3}
        x^{k}= x^{k-1} - \omega \frac{\l x^{k-1}, A_{i_k} \r - \hb^{k}_{i_k}}{\p A_{i_k} \p^2} A_{i_k}.
        \ee
    \EndFor  
  \end{algorithmic}
\end{algorithm}

\section{Convergence Analysis}\label{sec:convergence}

In order to prove the convergence of the two algorithms MREK \ref{alg:MREK} and ACEK \ref{alg:ACEK}, we next examine how the distance to any fixed least-squares solution changes.

To this end, we denote by $x^k_\ast=P_{H_\ik}(x^{k-1})$, where $H_\ik$ is the unperturbed hyperplane
from \eqref{eq:Hik}, given by
\be\label{eq:xkstar1}
x^k_\ast=x^{k-1}-\omega \frac{\l A_\ik ,x^{k-1} \r-b_\ik}{\|A_\ik\|^2}A_\ik,
\ee

\begin{proposition}
\label{prop:quasi-mon}
For any $x \in LSS(A; \hat{b})$ and for all $k \in \N$, we have for every iterate $x^k$ generated by
 the algorithm MREK \ref{alg:MREK} or ACEK \ref{alg:ACEK}, respectively 
and for any $\ik\in [m]$ 
\begin{itemize}
\item[(i)]
\be
\label{quasi-mon1-1}
\p x^k - x \p^2 = \p x^k_* - x \p^2 + \omega^2 \p \gamma_\ik \p^2,
\ee
\item[(ii)]
\be
\label{quasi-mon1-2}
\p x^k - x \p^2 = \p x^{k-1} - x \p^2 - \omega (2 - \omega) \frac{\left(\l A_\ik,x^{k-1}\r-b_\ik\right)^2}{\|A_\ik\|^2}
+ \omega^2 \p \gamma_\ik \p^2, 
\ee
\item[(iii)]
\be
\label{quasi-mon1-3}
\p x^k - x \p^2 \le \p x^{k-1} - x \p^2  
+ \omega^2 \p \gamma_\ik \p^2,
\ee
with $\gamma_\ik$ from \eqref{def:delta}.
\end{itemize}
\end{proposition}

\begin{proof} (i) Choose $x \in LSS(A; \hat{b})$ arbitrarily. Then $Ax=b$ and, in particular,
$x\in H_\ik$.  Since 
$x^k_\ast\in H_\ik$, Lemma \ref{lemma:shift} (see also (\ref{ACEK-3})) asserts
$x^{k}=x^k_\ast+ \omega \gamma_\ik $.  The orthogonality relation $\gamma_\ik \perp (x_{\ast}^{k}-x) \in H_{i_{k}}$
due to $\gamma_\ik =\delta_\ik A_\ik$ \eqref{def:delta},
immediately gives  $ \p x^k - x \p^2 = \p x^k_* - x \p^2 + \omega^2 \p \gamma_\ik \p^2.$\\
(ii) We will denote by $P^{\omega}_{H_{\ik}}$ the right hand side of (\ref{eq:xkstar1}), i.e.
\be
\label{eq10}
x^k_\ast = P^{\omega}_{H_{\ik}}(x^{k-1}) = x^{k-1}-\omega \frac{\l A_\ik ,x^{k-1} \r-b_\ik}{\|A_\ik\|^2} A_\ik.
\ee
If $S_{\ik} = \{ x: \l A_{\ik}, x \r = 0 \}$ denotes the corresponding vector subspace (see (\ref{eq:Hik})), and because $b_{\ik} = \l A_{\ik}, x \r$ then the application $P^{\omega}_{S_{\ik}}(z) = z-\omega \frac{\l A_\ik, z \r}{\|A_\ik\|^2} A_\ik$, which satisfies 
\be
\label{eq10p}
x^k_\ast - x = P^{\omega}_{S_{\ik}}(x^{k-1} - x),
\ee
has similar properties with $\varphi^{\alpha}_j$ from (\ref{n1}). Let also $P_{S_{\ik}}(z) = z- \frac{\l A_\ik, z \r}{\|A_\ik\|^2} A_\ik$. 
Then, from Lemma \ref{lem21}, (\ref{eq69}) applied to $P^{\omega}_{S_{\ik}}$ and $P_{S_{\ik}}$ we get (by also using the fact that the projection  $P_{S_{\ik}}$ is an idempotent operator)
$$
\p x^k_\ast - x \p^2 = \p P^{\omega}_{S_{\ik}}(x^{k-1} - x) \p^2 = \omega (2 - \omega) \left( \p P_{S_{\ik}}(x^{k-1} - x) \p^2 - \p x^{k-1} - x \p^2 \right) + 
$$
$$
\p x^{k-1} - x \p^2 = \omega (2 - \omega) \l P_{S_{\ik}}(x^{k-1} - x), x^{k-1} - x \r + (1-\omega (2 - \omega)) \p x^{k-1} - x \p^2 =  
$$
\be
\label{eq10pp}
\p x^{k-1} - x \p^2  - \omega (2 - \omega) \frac{\l A_\ik, x^{k-1} - x \r^2}{\|A_\ik\|^2}.
\ee
Then,  equation (\ref{quasi-mon1-2})  follows from (\ref{quasi-mon1-1}) and (\ref{eq10pp}).\\
(iii) It results directly from (\ref{quasi-mon1-2}) and the proof is complete.
\end{proof}

\begin{remark} Proposition \ref{prop:quasi-mon} (iii), together with Lemmata \ref{lem:gammaMREK} (ii)
and \ref{lem:gammaACEK} (ii) shows that the sequence $(x^k)_{k \in\N}$ generated by the MREK \ref{alg:MREK} or the ACEK
algorithm \ref{alg:ACEK} is \emph{quasi-F\'{e}jer of Type II}, see \cite[Def. 1.1]{Combettes2001}.
\end{remark}

The next Lemma is a special case of Lemma 3.1 in \cite{Combettes2001}. The corresponding simplified
proof is included for completeness. 
\begin{lemma}
\label{lem:Lemma3.1}
Let $(\alpha_k)_{k\in \N}\in \ell_+$ and $(\beta_k)_{k\in \N}\in \ell_+$ be two
nonnegative sequences, and $(\veps_k)_{k\in \N}\in \ell_+ \cap \ell^1$ satisfying
\be\label{eq:Lemma3.1}
\alpha_{k+1}=\alpha_{k}- \beta_k + \veps_k .
\ee
Then the following statements hold true.
\begin{itemize}
\item[(i)] $(\beta_k)_{k\in \N}\in \ell^1$. In particular $(\beta_k)_{k\in \N}\in \ell_{c_0}$,
\item[(ii)] $(\alpha_k)_{k\in \N}$ converges.
\end{itemize}
\end{lemma}
\begin{proof} (i) From \eqref{eq:Lemma3.1}, we have $\beta_k=\alpha_{k}- \alpha_{k+1} + \veps_k$.
Furthermore,
\[
\sum_{k=0}^n\beta_k=\sum_{k=0}^n(\alpha_{k}- \alpha_{k+1}) + \sum_{k=0}^n\veps_k
=\alpha_{0}- \alpha_{n+1} + \sum_{k=0}^n\veps_k < \alpha_{0} + \sum_{k=0}^n\veps_k,
\]
which yields $\sum_{k\in\N}\beta_k < \alpha_{0}+ \sum_{k\in\N} \veps_k<+\infty$. Hence
$(\beta_k)_{k\in \N}\in \ell^1$.
Now $\ell^1\subset \ell_{c_0}$, shows (i).\\
(ii) Now, both $(\veps_k)_{k\in \N}\in \ell_{c_0}$ and $(\beta_k)_{k\in \N}\in \ell_{c_0}$.
By \eqref{eq:Lemma3.1},
\[
|\alpha_{k+1}- \alpha_{k}|=|\veps_k -\beta_k | \le |\veps_k|+ |\beta_k |=\veps_k+\beta_k,
\]
with $(\veps_k+\beta_k)_{k\in \N}\in \ell^1$.
This shows that $(\alpha_k)_{k\in \N}$ is a Cauchy sequence\footnote[1]{An arbitrary 
sequence $(y^k)_{k\in\N}$ is Cauchy, if
$\|y^{k+1}-y^k\|\le a_k$ holds for all $k\in\N$ and $(a_k)_{k\in\N}\in\ell^1\cap \ell_+$ arbitrary.
Indeed, $\|y^{m+k}-y^m\|=\|\sum_{j=m}^{m+k-1}(y^{j+1}-y^{j})\|\le \sum_{j=m}^{m+k-1} \|y^{j+1}-y^{j}\|
\le \sum_{j=m}^{m+k-1} a_j=s_{m+k-1}-s_{m-1}$, with $s_n:=\sum_{j=1}^{n} a_j$. Now $(s_n)_{k\in\N}$ is Cauchy since it converges due to $(a_k)_{k\in\N}\in\ell^1$.}. 
Since $(\alpha_k)_{k\in \N}\in\ell_+\subset \R$ it also converges. 
\end{proof}

We are now ready to prove convergence of MREK, Algorithm \ref{alg:MREK}.
\begin{theorem}
\label{th:MREK}
Let $\alpha, \omega \in (0, 2)$. 
The sequence $(x^k)_{k \in\N}$ generated by the MREK,
Algorithm \ref{alg:MREK}, converges to a least-squares solution  in $LSS(A; \hat{b})$,
for any starting vector $x^0 \in \R^n$. 
\end{theorem}
\begin{proof} We split the proof into two parts, showing convergence of $(x^k)_{k\in\N}$, and convergence to a point in $LSS(A;\hat b)$, respectively.
\begin{enumerate}[(i)]
\item
Choose any $x \in LSS(A; \hat{b})$ and set
\[
\alpha_{k+1}=\p x^k - x \p^2, \quad 
\beta_k =  \omega (2 - \omega) \frac{\left(\l A_\ik,x^{k-1}\r-b_\ik\right)^2}{\|A_\ik\|^2},\quad
\veps_k= \omega^2 \p \gamma_\ik \p^2, 
\]
The above Lemma (see also (4.7) !!!) asserts convergence of $(\alpha_k)_{k\in\N}$ and $(\beta_k)_{k\in\N}\in \ell^1$, in view of
$\veps_k\in\ell^1$, due to Lemma \ref{lem:gammaMREK} (ii) and Prop. \ref{prop:quasi-mon} (ii) respectively. In view of \eqref{MREK-5}, we get
\be\label{eq:xkCauchy}
\p x^k - x^{k-1} \p^2 = \omega^2 \left\|  -  \frac{\l A_\ik,x^{k-1}\r-b_\ik}{\|A_\ik\|^2} A_\ik + \gamma_\ik \right\|^2 
\le \frac{2 \omega}{2 - \omega} \beta_k +2 \veps_k .
\ee
Now $(\frac{2 \omega}{2 - \omega} \beta_k +2 \veps_k )_{k\in\N}\in\ell^1$ implies that $(x^k)_{k\in\N}$ is a Cauchy sequence\footnote[2]{Argument as above.} and converges as well.
In particular, using again \eqref{MREK-5},
\be\label{eq:res_bhatk}
\p x^k - x^{k-1} \p^2 
=  \omega^2 \frac{\left(\l A_\ik,x^{k-1}\r-\hat b^k_\ik\right)^2}{\|A_\ik\|^2} \to 0.
\ee

\item
Assume that $x^k\to \ol{x}$.
We show that $\ol{x} \in LSS(A,\hat b)$.
Fix any $i\in[m]$. Due to the particular choice of $\ik$ in \eqref{eq:ik_MR},
we have
\begin{align*}
|\l A_i,x^{k-1}\r- b_i| - |r_i- y^k_i|  & \le |\l A_i,x^{k-1}\r- b_i - (r_i- y^k_i)|\\
 & = |\l A_i,x^{k-1}\r-\hat b^k_i|\\
 &\stackrel{\eqref{eq:ik_MR}}{\le} |\l A_\ik,x^{k-1}\r-\hat b^k_\ik|.
\end{align*}
Thus $|\l A_i,x^{k-1}\r- b_i|\to 0$, due to $|r_i- y^k_i|\to 0$ by Lemma \ref{lem:gammaMREK} (iii) and \eqref{eq:res_bhatk}, respectively.
Summarizing, we get $\lim_{k\to\infty}\|Ax^{k-1}-b\|=0=\|A\ol{x}-b\|$.
Thus, $\ol{x} \in LSS(A,\hat b)$.
\end{enumerate}
\end{proof}

The main result concerning convergence of ACEK, Algorithm \ref{alg:ACEK}, is stated next.
\begin{theorem}
\label{th:ACEK}
The sequence $(x^k)_{k \in\N}$ generated by ACEK, 
Algorithm \ref{alg:ACEK}, converges to a least-squares solution in $LSS(A; \hat{b})$,
for any starting vector $x^0 \in \R^n$. 
\end{theorem}
\begin{proof} Choose any $x \in LSS(A; \hat{b})$ and set
\[
\alpha_{k+1}=\p x^k - x \p^2, \quad 
\beta_k = \omega (2 - \omega) \frac{\left(\l A_\ik,x^{k-1}\r-b_\ik\right)^2}{\|A_\ik\|^2},\quad
\veps_k= \omega^2 \p \gamma_\ik \p^2. 
\]
The proof of convergence $x^k\to\ol{x}$ is identically to
the first part of the proof of Thm.~\ref{th:MREK}, with the only difference that we
have $(\veps_k)_{k\in\N}\in\ell^1$ due to Lemma \ref{lem:gammaACEK}, (ii). Moreover 
\be\label{eq:clean_res}
\l A_\ik,x^{k-1}\r- b_\ik\to 0
\ee 
holds. The selection of $\ik$ in \eqref{eq:ik_AC} ensures $[m] \subset (i_{k})_{k \in \N}$. 
This, together with \eqref{eq:clean_res}, implies $A \bar{x} = b$ and completes the proof.
\end{proof}


\section{Conclusions}
\label{sec:conclusion}

We consider an inconsistent system of linear equations and our goal is to find the least squares (LS) solution.
It is known that the Kaczmarz method does not converge to the LS solution in this case. In its randomized form the Kaczmarz method converges with a radius proportional the magnitude of the largest entry of the noise in the system. Convergence to the LS solution can be achieved if step lengths converging to zero are used. Unfortunately this significantly compromises convergence speed. A different approach is adopted by the extended Kaczmarz (EK) algorithm. In both randomized and deterministic forms, the methods alternates between projections on hyperplanes defined by the rows of the matrix and projections on the subspace orthogonal to the matrix range defined by the matrix columns. By this procedure the method iteratively builds a corrected right hand side which is then simultaneously exploited by Kaczmarz steps applied to a corrected system. The randomized extended Kaczmarz (REK)
converges in expectation to the least squares solution and convergence rates can be obtained,  as recently shown by Zouzias and Freris. For deterministic control strategies however, the convergence was still open when alternating between row and columns updates. We close this gap by showing convergence to the LS solution.

\vspace*{1cm}

\bibliographystyle{alpha}

\end{document}